\documentclass{article}

\usepackage{amsmath}
\usepackage{amsthm}
\usepackage{amssymb}
\usepackage{dsfont}
\usepackage{calc}
\usepackage{eufrak} 
\usepackage{enumitem} 
\usepackage{authblk}
\usepackage{theoremref}
\usepackage{verbatim}

\begin{document}
\bibliographystyle{alpha}
\newtheorem{theorem}{Theorem}
\newtheorem{lemma}{Lemma}
\newtheorem{lem}{Lemma}[section]
\newtheorem{obs}{Observation}[section]
\newtheorem{definition}{Definition}
\newtheorem{df}{Definition}[section]
\newtheorem{proposition}{Proposition}
\newtheorem{observation}{Observation}
\newtheorem{question}{Question}
\newtheorem{remark}{Remark}
\newtheorem{property}{Property}
\newtheorem{corollary}{Corollary}
\newcounter{casenum}
\newenvironment{caseof}{\setcounter{casenum}{1}}{\vskip.5\baselineskip}
\newcommand{\case}[2]{\vskip.5\baselineskip\par\noindent {\bfseries Case \arabic{casenum}:} #1\\#2\addtocounter{casenum}{1}}

\makeatletter
\newsavebox{\@brx}
\newcommand{\llangle}[1][]{\savebox{\@brx}{\(\m@th{#1\langle}\)}%
  \mathopen{\copy\@brx\kern-0.5\wd\@brx\usebox{\@brx}}}
\newcommand{\rrangle}[1][]{\savebox{\@brx}{\(\m@th{#1\rangle}\)}%
  \mathclose{\copy\@brx\kern-0.5\wd\@brx\usebox{\@brx}}}
\makeatother

\title{Computability, orders, and solvable groups}
\author{Arman Darbinyan}
\date{}

\affil{\emph{}}

\maketitle

\begin{abstract}
	The main objective of this paper is the following two results. (1) There exists a computable bi-orderable group that does not have a computable bi-ordering; (2) There exists a bi-orderable, two-generated recursively presented solvable group with undecidable word problem. Both of the groups can be found among two-generated solvable groups of derived length $3$.
	
	(1) answers a question posed by Downey and Kurtz; (2) answers a question posed by Bludov and Glass in Kourovka Notebook.
	
	One of the technical tools used to obtain the main results is a computational extension of an embedding theorem of B. Neumann that was studied by the author earlier. In this paper we also compliment that result and derive new corollaries that might be of independent interest.
\end{abstract}






\section{Main results}
A presentation $G=\langle X \mid R \rangle$ of a countable group is called \emph{recursively enumerated} if the sets $X$ and $R \subseteq (X \cup X^{-1})^*$ are recursively enumerated or, equivalently, if the set $\{ u \in (X \cup X^{-1})^* \mid u=_G 1 \}$ is recursively enumerated. It is said that $G=\langle X \rangle$ is  a \emph{computable group} with respect to the recursively enumerated generating set $X$ if the set $\{ u \in (X \cup X^{-1})^* \mid u=_G 1 \}$ is recursive. The concept of computable groups was introduced by Rabin \cite{rabin} and Mal'cev \cite{mal'cev}.

Finitely generated groups that are computable with respect to a finite generating set are called groups with decidable word problem. A well-known property of groups with decidable word problem is that decidability of the word problem does not depend on the choice of finite generating set, hence it is an intrinsic property of the group. This is in contrast with the general case of countable groups when the property of being computable depends on the choice of the generating set.

To formulate the first main theorem, we introduce the following definition, which is a weaker form of left- and bi- orderings on groups.
\begin{definition}[p-order]
For a given group $G$, we say that a binary relation $\preceq$ on $G$ is a \emph{p-order}	 if 
\begin{itemize}
	\item it is antisymmetric, i.e. $g\preceq h$ and $h \preceq g$ imply $g=h$;
	\item $1 \preceq g$ implies $g^{-1} \preceq 1$;
	\item $1 \preceq g$ implies $1 \preceq g^n$ for any $n \in \mathbb{N}$.
\end{itemize}
If a group $G$ possesses a p-order relation, then we say that $G$ is \emph{p-orderable}.
\end{definition}

Let us recall that $G$ is a \emph{left-orderable} group if there exists a liner order on $G$ that is invariant under the left multiplication. If, in addition, the order on $G$ is also invariant under the right multiplication, then the group $G$ is called \emph{bi-oderable}.
Note that left- and bi- orders on groups are  p-order relations.

\begin{definition}
The linear order $\preceq$ on $G=\langle X \rangle$ is said to be computable with respect to the generating set $X=\{x_1, x_2, \ldots \}$ if the set $\{ (u, v) \in (X\cup X^{-1})^* \times (X\cup X^{-1})^* \mid u \preceq v \} $ is recursive.
\end{definition}

Our next main result is the following.

\begin{theorem}
	\label{theorem p-Downey-Kurtz}
	There exists a two-generated bi-orderable computable group  that does not embed in any countable group with a computable p-order relation. Moreover,  can be chosen to be a solvable group of derived length $3$.
\end{theorem}
The question of Downey and Kurtz asked in \cite{downey-remmel} as Question 6.12 (ii) is as follows.
\begin{quote}
	 Is every computable orderable group isomorphic to computably orderable group?                                                                                                                                                                                                                                            
\end{quote}

For the case of left-orderable groups a negative answer was recently obtained by Harrison-Trainor in \cite{harrison-trainor}. However, as the author mentioned in the abstract of \cite{harrison-trainor}, the more general case of bi-orderable groups is left as open. 

From Theorem \ref{theorem p-Downey-Kurtz} we immediately get the answer to the question of Downey and Kurtz for the general case of bi-orderable groups.
\begin{corollary}
	\label{cor-downey-kurtz}
	There exists a computable bi-orderable group $G$, which does not have a presentation with computable bi-order. $G$ can be chosen to be two-generated solvable group of derived length $3$.
\end{corollary}

In \cite{bludov-glass-1}, Bludov and Glass showed that every left-orderable computable group embeds into a finitely presentable left-orderable group with decidable word problem.
The combination of these two results with Theorem \ref{theorem p-Downey-Kurtz} immediately leads to yet another strengthening of the result of Harrison-Trainor as follows.

\begin{corollary}
	There exists a finitely presentable left-orderable group with decidable word problem without computable left-order.
\end{corollary}


Addressing the question of Downey and Kurtz, Solomon showed in \cite{solomon-1} that for the class of abelian groups (i.e., for solvable groups of derived length $1$) the answer to the question is positive, i.e. every bi-orderable computable abelian group possesses a presentation with computable bi-order. This result in combination with Corollary \ref{cor-downey-kurtz} leads to the following question.
\begin{question}
Is it true that every computable metabelian bi-orderable group possesses a presentation with computable bi-order?
\end{question}

In Kourovka notebook \cite[ Question 17.28]{kourovka}, Bludov and Glass asked the following question.
\begin{quote}
	Is there a solvable left-orderable group with undecidable word problem?
\end{quote}

The next theorem answers this question.

\begin{theorem}
\label{theorem bludov-glass}
	There exists a two-generated recursively presented bi-orderable solvable group  of derived length $3$ with undecidable word problem.
\end{theorem}

\subsection{Computability and group embeddings}
The main technical tool of our paper is the following embedding theorem, which, in a weaker form and using a slightly different language, can be found in \cite{darbinyan - embedding}, and is based on a modified version of B. Neumann's construction from \cite{neumann-embedding}.

\begin{theorem}
\label{theorem-embedding}
Let $H= \langle X \rangle$ be a group with countable generating set $X=\{x_1, x_2, \ldots\}$. Then there exists an embedding $\Phi_X: H \hookrightarrow G$ into a two-generated group $G=\langle f, c \rangle$ such that the following takes place.
\begin{enumerate}[label=(\alph*)]
    \item \label{c}There exists a computable map $\Phi_X: i \mapsto \{f^{\pm 1}, s^{\pm 1}\}^*$ such that $\Phi_X$ represents the element $\Phi_X(x_i)$ in $G$;
	\item \label{d} If $H$ is a solvable group of derived length $l$, then $G$ is a solvable group of derived length $l+2$;
	\item \label{a} $G$ has a recursive presentation if and only if $H$ has a recursive presentation with respect to the generating set $X$;
	\item \label{b} $G$ has decidable word problem if and only if $H$ is computable with respect to the generating set $X$;
		\item \label{f}If $H$ is a computable group with respect to the generating set $X$, then the membership problem for the subgroup $\Phi_X(H) \leq G$ is decidable, i.e. there exists an algorithm that for any $g\in G$ decides whether or not $g \in \Phi_X(H)$;
	\item \label{e}If $H$ is  left- or bi- orderable, then so is $G$ and $G$ continues the order on $H$. Moreover, if with respect to the generating set $X$ there is a computable order on $H$, then $G$ has a computable order as well that continues the order on $H$;
	\item \label{dd} There exist $N_1, N_2 \leq G$ such that $\Phi_X(H) \vartriangleleft N_1 \vartriangleleft N_2 \vartriangleleft G $, hence $\Phi_X(H)$ is a subnormal subgroup of $G$;
	\end{enumerate}	
\end{theorem}
\begin{remark}
	We would like to mention that the computational properties of the embedding $\Phi_X$ essentially depend on the choice of the generating set $X$. In fact, in applications in Theorems \ref{theorem p-Downey-Kurtz} and \ref{theorem bludov-glass}, Theorem \ref{theorem-embedding} is applied on groups isomorphic to $\bigoplus_{i=1}^{\infty} \mathbb{Z}_i$, where $\mathbb{Z}_i$ are copies of $\mathbb{Z}$. If not particular presentations of $\bigoplus_{i=1}^{\infty} \mathbb{Z}_i$, then the corresponding embeddings would not lead to desired computational properties.
\end{remark}

The next corollaries follow straightforwardly from Theorem \ref{theorem-embedding}.
\begin{corollary}

\label{corollary-characterization of computable groups}
    A countable group $H$ is computable if and only if it is a (subnormal) subgroup of a finitely-generated group with decidable word problem such that the membership problem for the subgroup $H$ is decidable in the large group. 
\end{corollary}

\begin{corollary}
\label{corollary-characterization of computable orders}
    A countable group $H$ has a computable  left- or bi- ordering if and only if it is a (subnormal) subgroup of a finitely-generated group with computable  left- or bi- order, respectively, such that the membership problem for the subgroup $H$ is decidable in the large group. Moreover, for any fixed computable order on $H$ we can assume that the large group continues the order on $H$.
\end{corollary}
\begin{remark}
For a group $H$ both the properties of being computable and having computable  left- or bi- orders depends on the presentation of the group. One of the noteworthy aspects of Corollaries \ref{corollary-characterization of computable groups} and \ref{corollary-characterization of computable orders} is that they reformulate these computability properties for $H$ in terms that do not depend on the presentation of the group.
\end{remark}

\section{Proof of Theorem \ref{theorem p-Downey-Kurtz}}

\begin{lem}
\label{lemma-order computability reduction to finite case}
Let $G$ be a finitely generated group with decidable word problem. Then $G$ has a presentation with computable (p-, left-, or bi-) order if and only if that order is computable with respect to any finite generating set.
\end{lem}
\begin{proof}

Let us assume that $G=\langle X \rangle$, $X=\{x_1^{\pm 1}, x_2^{\pm 1}, \ldots \}$, such that for a given (p-, left-, or bi-)  order $\preceq_G$ on $G$, the set $\{ (u, v) \subseteq X^* \times X^* \mid u \preceq_G v \} $ is recursive. Then, by Corollary \ref{corollary-characterization of computable orders}, there exists a two-generated group $H=\langle a, b \rangle$ such that $G\leq H$,  $H$ has order $\preceq_H$ that continues the order $\preceq_G$ and is computable, and the membership problem for $G \leq H$ is decidable.

Now let us assume that $G$ is also generated by a finite set $S=\{s_1^{\pm 1}, \ldots, s_n^{\pm 1}\}$. Let $u_1,  u_2, \ldots, u_n \in \{a^{\pm 1}, b^{\pm 1} \}^*$ be such that they represent the elements $s_1, \ldots, s_n$ in $G$. Let $\phi: S^* \rightarrow X^*$ be the map induces by $s_1 \mapsto u_1$, \ldots, $s_n \mapsto u_n$. Then, 
for any $w_1, w_2 \in S^*$, $w_1 \preceq_G w_2$ is equivalent to $\phi(w_1) \preceq_H \phi(w_2)$. Therefore, since $\preceq_H$ is a computable order, we get that $\{ (u, v) \subseteq S^* \times S^* \mid u \preceq_G v \} $ is recursive as well. Since $S$ is an arbitrarily chosen generating set, one direction of the lemma is proved. The other direction is trivial.
\end{proof}

\begin{lem}
\label{lemma second about computable order}
If a finitely generated group has a computable p-order, then the induced p-order on any finitely generated subgroup is also computable	
\end{lem}
\begin{proof}
	Indeed, let $G=\langle S \rangle$ be a finitely generated group with a computable p-order relation $\preceq_G$, i.e.
	the set $\{ u \in (S \cup S^{-1})^* \mid 1\preceq_G u \}$ is recursive. Let $H=\langle X \rangle \leq G$, $|X|< \infty$, be a finitely generated subgroup of $G$ with the induced p-order relation $\preceq_H$ (i.e. for $h_1, h_2 \in H$, $h_1 \preceq_H h_2$ if and only if $h_1 \preceq_G h_2$.) Then $\preceq_H$ is a computable relation with respect to the generating set $X$, because every word $w\in (X\cup X^{-1})^*$ can be computably rewritten as a word $v \in (S \cup S^{-1})^*$ such that $w$ and $v$ represent the same element of $H$. The last observation means that recursiveness of $\{ u \in (S \cup S^{-1})^* \mid 1\preceq_G u \}$ implies recursiveness of $\{ u \in (X \cup X^{-1})^* \mid 1\preceq_H u \}$.
\end{proof}

Let us fix a  recursively enumerable and recursively inseparable pair $(\mathcal{M}, \mathcal{N})$ of subsets of $\mathbb{N}$. (For the existence of such a pair see \cite{smullyan, shoenfield-logic}.)
Let $\mathcal{M}=\{m_1, m_2, \ldots \}$ and $\mathcal{N}=\{n_1, n_2, \ldots \}$ such that the enumerations are recursive.

Let $\mathcal{P}=\{p_1, p_2, \ldots \}$ be the set of primes listed in its natural order.

Now, let us consider the abelian group $A \simeq \bigoplus_{i=1}^{\infty} \mathbb{Z}_i$ given by the presentation
\begin{align}
\label{equation of the presentation for downey-kurtz}
    A=\langle a_i, i=1, 2, \ldots \mid [a_i, a_j]=1,  a_{2n_i}= a_{2n_i-1}^{p_i}, a_{2m_i}= a_{2m_i-1}^{-p_i}, i, j\in \mathbb{N} \rangle.
\end{align}

\begin{lem}
The presentation \eqref{equation of the presentation for downey-kurtz} is computable.
\end{lem}
\begin{proof}
Indeed, any element of $A$ can be presented as a finite length word in the alphabet $\{a_1^{\pm 1}, a_2^{\pm 2}, \ldots\}$, which, on its own turn, can be computationally transformed into another, reduced, word representing the same element which is of the form
\begin{align}
\label{equation intermediate}
   a_{i_1}^{\epsilon_1} a_{i_2}^{\epsilon_2} \ldots a_{i_n}^{\epsilon_n}, 
\end{align}
where $\epsilon_i \in \{\pm 1\}$. If this word is not empty and it represents the trivial element, then it must contain a subword of one of the following forms: $(a_{2n_i}a_{2n_i-1}^{-p_i})^{\pm 1}$ or $(a_{2m_i} a_{2m_i-1}^{p_i})^{\pm 1}$. Since the set of primes is recursive and $\mathcal{M}$ and $\mathcal{N}$ are recursively enumerated, the set of subwords $\{(a_{2n_i}a_{2n_i-1}^{-p_i})^{\pm 1}, (a_{2m_i} a_{2m_i-1}^{p_i})^{\pm 1} \mid i \in \mathbb{N} \}$ is recursive, hence, they can be computationally detected. Therefore, in order to check whether or not a word of the form \eqref{equation intermediate} represents a trivial element of the group $A$, we can simply find the mentioned subwords and remove  them until this procedure can not be continued. Since it is a computable procedure, the presentation \eqref{equation of the presentation for downey-kurtz} is computable.
\end{proof}

Now, let $G$ be the group obtained after embedding the group $A$ into a two-generated group according to the embedding from Theorem \ref{theorem-embedding} (we consider the embedding with respect to the presentation \eqref{equation of the presentation for downey-kurtz}.) Let the embedding be $\Phi: A \hookrightarrow G$.

By the properties of the embedding of Theorem \ref{theorem-embedding}, $G=\langle x^{\pm 1}, y^{\pm 1}\rangle$ will have decidable word problem (part \ref{b}), will be bi-orderable (part \ref{e}), and will be solvable of derived length $3$ (part \ref{d}).

However, below  we show that $G$ does not possess any computable p-order with respect to any presentation. First, note that by Lemma \ref{lemma-order computability reduction to finite case}, if $\preceq_G$ is a computable preorder on $G$ with respect to some presentation, then it is a computable p-order with respect to the generating set $\{x^{\pm 1}, y^{\pm 1}\}$, i.e. the set
\begin{align*}
	\mathcal{O}=\{ (u, v) \in \{x^{\pm 1}, y^{\pm 1}\}^* \times \{x^{\pm 1}, y^{\pm 1}\}^* \mid u \preceq_G v \}
\end{align*}
is recursive. 
It follows from part \ref{c} of Theorem \ref{theorem-embedding} that there exist a computable map $\phi: \mathbb{N} \rightarrow \{x^{\pm 1}, y^{\pm 1}\}^* $ such that $\phi(i) =  u_i \in \{x^{\pm 1}, y^{\pm 1}\}^*$ is a word representing the element $\Phi(a_i) \in G$. Therefore, the set
\begin{align*}
	\mathcal{L} = \{ n \in \mathbb{N} \mid [u_{2n} \preceq_G 1 ~\& ~u_{2n-1} \preceq_G 1] \mbox{OR} [ 1 \preceq_G u_{2n} ~\& ~1 \preceq_G u_{2n-1} ] \}
\end{align*}
is recursive as well. However, since $\preceq_G$ is a p-order relation, we have that $\mathcal{N} \subseteq \mathcal{L}$ and $\mathcal{M} \cap \mathcal{L} = \emptyset$. We get to a contradiction with the assumption that the pair $(\mathcal{M}, \mathcal{N})$ is recursively inseparable. Therefore, $G$ does not possess a computable p-order with respect to any presentation.

Now, let us show that the group $G$ cannot be embedded into a countable group with a presentation possessing a computable p-order. Indeed, by contradiction assume that $H$ is a countable group with a computable preorder $\preceq_H$ in which $G$ is embedded. It follows from Theorem \ref{theorem-embedding} and  Lemma \ref{lemma-order computability reduction to finite case}  that without loss of generality we can assume that $H$ is finitely generated and the computability of $\preceq_H$ is with respect to (any) finite generating set. However, since $G$ does not have any computable p-order, by Lemma \ref{lemma second about computable order} we immediately reach a contradiction.

\section{Proof of Theorem \ref{theorem bludov-glass}}
Let us fix a recursively enumerated non-recursive set $\mathcal{N}=\{n_1, n_2, \ldots \}$. Let us consider the following presentation of the group $\bigoplus_{i=1}^{\infty} \mathbb{Z}_i$:
\begin{align}
	\label{equation of the presentation for bludov-glass}
    A=\langle a_i, i=1, 2, \ldots \mid [a_i, a_j]=1,  a_{2n_i}= a_{2n_i-1}, i, j\in \mathbb{N} \rangle.
\end{align}
Since $\mathcal{N}$ is recursively enumerable, the presentation \eqref{equation of the presentation for bludov-glass} is recursive. However, since $a_{2n}=a_{2n-1}$ is equivalent to $n\in \mathcal{N}$ and $\mathcal{N}$ is not recursive, we get that $\eqref{equation of the presentation for bludov-glass}$ is not computable. Now, let  $A$ be embedded into a two-generated group $G$ according to Theorem \ref{theorem-embedding}. Then, $G$ is solvable of derived length $3$ and is bi-orderable. By part \ref{a} of Theorem \ref{theorem-embedding}, $G$ is recursively presentable, by part \ref{b}, the word problem in $G$ is undecidable.

\section{Proof of Theorem \ref{theorem-embedding}}

The proof of Theorem \ref{theorem-embedding} uses the embedding construction  described in \cite{darbinyan - embedding}. Moreover, the proof is in a way a commentary on the proof of the main theorem from \cite{darbinyan - embedding}. There are also differences in notations  compared with the exposition from \cite{darbinyan - embedding}. For example, the embedding $\Phi_X$ described below is the embedding $\phi$ from \cite{darbinyan - embedding}, the map $f$ defined below is $c$ in \cite{darbinyan - embedding}, $f_i$ is $b^{(i)}$, etc.

Let us recall the definition of wreath product. Given two groups $A$ and $B$, the base subgroup $A^{B}$ is the set of functions from $B$ to $A$ with pointwise multiplication and the group $B$ acts on $A^B$ from the left by automorphisms, such that for $f\in A^B$ and $b \in B$, the resulting function $bf$ is given by 
$$ (b f)(x) = f(xb), \forall x \in B. $$
For the convenience we will denote $bf$ by $f^b$.

\begin{df}
The wreath product $A Wr B$ is defined as a semidirect product $ A^B \rtimes B$, with the multiplication $(f_1b_1)(f_2 b_2) = f_1f_2^{b_1} b_1b_2$ for all $ f_1, f_2 \in A^B, ~b_1, b_2 \in B$. 
\end{df}

Let $H=\langle X \rangle$, where $X=\{x_1, x_2, \ldots\}$, be the initial group. In order to construct the group from $G$ of Theorem \ref{theorem-embedding}, first, we will describe an embedding of the group $H$ into a subgroup $L$ of the full wreath product $H \wr \langle z \rangle $, where $\langle z \rangle$ is an infinite cyclic group generated by $z$, as follows. Define the functions $f_i: \langle z \rangle \rightarrow H $, $i=1, 2, \ldots$, as
\begin{align*}
   f_i(z^n)= \left\{
                      \begin{array}{ll}
                       x_i  & \mbox{if $n > 0$ ,}\\
                       1 & \mbox{otherwise.}
                     \end{array}
                    \right.
\end{align*}
Note that 
\begin{align}
\label{eq-evidence}
   [z, f_i](z^n)= ( z f_i z^{-1}f_i^{-1} )(z^n) = \left\{
                      \begin{array}{ll}
                       x_i  & \mbox{if $n= 0$ ,}\\
                       1 & \mbox{otherwise.}
                     \end{array}
                    \right.
\end{align}
Denote  $L=\langle z, f_1, f_2, \ldots \rangle \leq H \wr \langle z \rangle$. Since $x_i \mapsto f_i$, $i\in \mathbb{N}$, embeds $H$ into $L$, we can conclude with the following.
\begin{obs}
\label{obs-41}
	The group $L$ is computable with respect to the generating set $\{ z, f_1, f_2, \ldots \} \leq H \wr \langle z \rangle$ if and only if $H$ had decidable word problem.
\end{obs}
If we regard the maps $f_i: \langle z \rangle \rightarrow H $ as elements of the wreath product $H \wr \langle z \rangle $, then it follows from \eqref{eq-evidence} that the following observation takes place.
\begin{obs}
\label{observation 1}
 The map $x_i \mapsto  [z, f_i]$, $i\in \mathbb{N}$, induces an isomorphism between $H$ and $ \langle [z, f_i] \mid i \in \mathbb{N} \rangle < L$.
 \end{obs}

The next step is to embed the group $L$ into a two generated subgroup $G=\langle f, s \rangle$ of the full wreath product $L \wr \langle s \rangle < H \wr\langle z \rangle \wr \langle s \rangle$, where $\langle s \rangle$ is another infinite cyclic group generated by $s$, as follows. Define the function $f: \langle s \rangle \rightarrow L $ as
\begin{align*}
   f(s^n)= \left\{
                      \begin{array}{ll}
                       z  & \mbox{if $n=1$ ,}\\
                       f_i  & \mbox{if $n=2^i$, $i\in \mathbb{N}$ ,}\\
                       1 & \mbox{otherwise.}
                     \end{array}
                    \right.
\end{align*}
Then 
\begin{align*}
   [f, f^{s^{2^i-1}}](s^n)=(ff^{s^{2^i-1}}f^{-1}f^{-s^{2^i-1}})(s^n)= \left\{
                      \begin{array}{ll}
                       [z, f_i]  & \mbox{if $n=1$ ,}\\
                       1 & \mbox{otherwise.}
                     \end{array}
                    \right.
\end{align*}
Therefore, by Observation \ref{observation 1}, the map $x_i \mapsto [f, f^{s^{2^i-1}}]$ induces an embedding of the group $H$ into $G=\langle f, s \rangle \leq L \wr \langle s \rangle $. Thereby, we are ready to define the embedding $\Phi_X: H \hookrightarrow G$ as the one induced by the map $ x_i \mapsto [f, f^{s^{2^i-1}}]$, $i\in \mathbb{N}$. Since this map is computable, it implies part \ref{c} of Theorem \ref{theorem-embedding}. Also, part \ref{d} of Theorem \ref{theorem-embedding} follows from the observation that $G$ is a subgroup of $H \wr \langle z \rangle \wr \langle s \rangle$.

Note that the final group $G$ depends on the initial choice and enumeration of the generating set of $H$. This property is of essential importance for our applications, as the initial group $H$ that we consider in applications is isomorphic to $\bigoplus_i \mathbb{Z}_i$, and only for its particular presentation it leads to the desirable computational properties for the end group $G$.

Note that every word $w$ from $\{f^{\pm 1},  s^{\pm 1}\}^*$ can be computationally rewritten in the form
\begin{align}
\label{eq rewriting 1}
    w'=(f^{s^{\gamma_1}})^{\beta_1}(f^{s^{\gamma_2}})^{\beta_2}\ldots (f^{s^{\gamma_n}})^{\beta_n} s^{\delta},
\end{align}
where  $\delta$, $\gamma_i$, $\beta_i$, $i=1, \ldots n$, are some integers, such that $w'$ represents the same element of $G$ as $w$.

\begin{lem}[Lemma 4, \cite{darbinyan - embedding}]
\label{lemma-word problem}
The word $w'$ from \eqref{eq rewriting 1} represents the trivial element in $G$ if and only if  the following conditions take place.
\begin{itemize}
    \item $\delta=0$,
    \item $(f^{s^{\gamma_1}})^{\beta_1}(f^{s^{\gamma_2}})^{\beta_2}\ldots (f^{s^{\gamma_n}})^{\beta_n}  $, regarded as a map $\langle s \rangle \rightarrow L$, takes trivial values on each point $s^{\mu}$, where $\mu$ ranges from $-3\max\{ \|\gamma_1 \|, \ldots, \|\gamma_n\|\}$ to $3\max\{ \|\gamma_1 \|, \ldots, \|\gamma_n\|\}$,
    \item $\sum_{j\in B_i} \beta_j=0$ for $i=1, \dots, n$, where $B_i=\large\{k\in \{1, \ldots, n\}\mid \gamma_j=\gamma_i \large \}$.
\end{itemize}

\end{lem}
~\\
Note that for each $\gamma, \beta \in \mathbb{Z}$, $(f^{s^{\gamma}}): \langle s \rangle \rightarrow L$ is computable, because for every $\mu \in \mathbb{Z}$ we have
\begin{align*}
    (f^{s^{\gamma}})^{\beta}(s^{\mu})=  (f(s^{\gamma+{\mu}})(s^{\mu}))^{\beta}=\left\{
                      \begin{array}{ll}
                       f_i^{\beta}  & \mbox{if $\gamma+\mu=2^i$, $i\in \mathbb{N}$,}\\
                       z^{\beta}  & \mbox{if $\gamma+\mu=1$,}\\
                       1 & \mbox{otherwise.}
                     \end{array}
                    \right.
\end{align*}

Therefore, Lemma \ref{lemma-word problem} implies that the word problem in $G$ can be computationally reduced to the computability property of the group $L$ with respect to the generating set $\{z, f_1, f_2, \ldots\}$. This observation combined with Observation \ref{obs-41} implies part \ref{b} of Theorem \ref{theorem-embedding}. Moreover, it implies that if $G$ has recursive presentation, then so does $L$ with respect to $L=\langle z, f_1, f_2, \ldots \rangle$ hence, by Observation \ref{obs-41}, also $H$ with respect to $X$. Combined with the fact that $\Phi_X$ is a computable map, it also implies the inverse. Thus part \ref{a} of Theorem \ref{theorem-embedding} is confirmed as well.


Proof of part \ref{f} of Theorem \ref{theorem-embedding} is based on the following two lemmas from \cite{darbinyan - embedding}. We refer to \cite{darbinyan - embedding} for more details.
\begin{lem}[Lemma 6, \cite{darbinyan - embedding}]
\label{lem 4}

The word $w'$ from \eqref{eq rewriting 1} represents an element from $\Phi_X(H)$ if and only if
\begin{itemize}
    \item $\delta=0$,
    \item $(f^{s^{\gamma_1}})^{\beta_1}(f^{s^{\gamma_2}})^{\beta_2}\ldots (f^{s^{\gamma_n}})^{\beta_n}  $, regarded as a map $\langle s \rangle \rightarrow L$, takes trivial values on each point $s^{\mu}$, where $\mu \neq 1$ and ranges from $-3\max\{ \|\gamma_1 \|, \ldots, \|\gamma_n\|\}$ to $3\max\{ \|\gamma_1 \|, \ldots, \|\gamma_n\|\}$,
    \item $\sum_{j\in B_i} \beta_j=0$ for $i=1, \dots, n$, where $B_i=\large\{k\in \{1, \ldots, n\}\mid \gamma_j=\gamma_i \large \}$, and
    \item $(f^{s^{\gamma_1}})^{\beta_1}(f^{s^{\gamma_2}})^{\beta_2}\ldots (f^{s^{\gamma_n}})^{\beta_n} (s) \in \langle~ [z,f_i] ~\mid~ i \in \mathbb{N} \rangle$.
\end{itemize}
\end{lem}
\begin{lem}[Lemma 7, \cite{darbinyan - embedding}]
$\bar{f}:= (f_1^{z^{\eta_1}})^{\xi_1} (f_2^{z^{\eta_2}})^{\xi_2} \ldots  (f_m^{z^{\eta_m}})^{\xi_m}z^{\eta} \in L$ belongs to $\langle [z, f_i] \mid i \in \mathbb{N} \rangle \leq L$ if and only if
$\eta = 0$ and $\bar{f} (z^{\pm 1})=\bar{f} (z^{\pm 2}) = \ldots = \bar{f} (z^{\pm \eta_0}) =1 $, where $\eta_0 = \max \{|\eta_1|, \ldots, |\eta_m| \}$.
\end{lem}
~\\

Part \ref{e} of Theorem \ref{theorem-embedding} is a direct consequence of the following lemma.
\begin{lem}
\label{lemma-computable order}

Let $K$ be a finitely generated subgroup of a wreath product $A \wr B$ of two left- or bi-  orderable groups such that for each $f \in A^B$ and $b \in B$, $fb \in K$ implies that $supp (f)$ is well-orderable with respect to an order on $B$. Then the group $K$ is correspondingly left- or bi-  orderable.

Moreover, if $K$ has decidable word problem, $A$ and $B$ possess computable orders $\preceq_A$ and $\preceq_B$ respectively, and the set  $ \big\{ (f_1 b, f_2b) \in K \times K \mid f_1(x_0) \prec_A f_2(x_0),  x_0=\min\{x \in  supp(f_1 f_2^{-1})\} \big\}$ is recursive, then $K$ has a computable order.
\end{lem}

\begin{proof}
Define $f_1b_1 \prec_K f_2b_2$ if either $b_1 \prec_B b_2$ or $b_1=b_2$ and $f_1 \neq f_2$ and $f_1(x_0) \prec_A f_2(x_0)$, where $x_0=\min\{x \in  supp(f_1 f_2^{-1})\} $. It is straightforward to show that  the order $\prec_K$ is a bi-order in case $\prec_A$ and $\prec_B$ are bi-orders on $A$ and $B$, respectively. See  \cite{neumann-embedding} for details. The case for left-orders can be shown identically.

Proof of the second statement: As for the first part, define $f_1b_1 \prec_K f_2b_2$ if either $b_1 \prec_B b_2$ or $b_1=b_2$ and $f_1 \neq f_2$ and $f_1(x_0) \prec_A f_2(x_0)$, where $x_0=\min\{x \in  supp(f_1 f_2^{-1})\} $. Therefore, since $K$ and  $\preceq_B$ are computable and, by our assumption, $ \big\{ (f_1 b, f_2b) \in K \times K \mid f_1(x_0) \prec_A f_2(x_0),  x_0=\min\{x \in  supp(f_1 f_2^{-1})\} \big\}$ is recursive, we get that $\preceq_K$ is a computable order.
\end{proof}
Now, part \ref{e} of Theorem \ref{theorem-embedding} follows from Lemma \ref{lemma-computable order} and the observation that the groups $L \leq H \wr \langle z \rangle$ and correspondingly $G=\langle f, s \rangle \leq L \wr \langle s \rangle $ satisfy the assumptions from Lemma \ref{lemma-computable order}.\\

Finally, for part \ref{f} of Theorem \ref{theorem-embedding}, note that we have 
\begin{align*}
    \Phi_X(H) = \langle [f, f^{s^{2^i-1}}] \mid i\in \mathbb{N} \rangle \vartriangleleft \langle [f, f^{j} f^{s^{2^i-1}} f^{-j}] \mid i, j\in \mathbb{N} \rangle \vartriangleleft \langle f^{s^i} \mid i \in \mathbb{N} \rangle  \vartriangleleft \langle f, s \rangle = G.
\end{align*}
Therefore, one can take $N_1 =  \langle [f, f^{j} f^{s^{2^i-1}} f^{-j}] \mid i, j\in \mathbb{N} \rangle$ and $N_2=\langle f^{s^i} \mid i \in \mathbb{N} \rangle$.
~\\
~\\
\noindent
\textbf{Acknowledgements:} I would like to thank Markus Steenbock for many motivating conversations. The paper was written during my postdoctoral year at \'Ecole Normale Sup\'erieure in Paris, where I was supported by ERC-grant GroIsRan no.725773 of Anna Erschler.


~\\
~\\
~\\
  A. Darbinyan, \textsc{Department of Mathematics, Mailstop 3368,
Texas A\&M University,
College Station, TX, USA}\par\nopagebreak
  \textit{E-mail address}: \texttt{arman.darbin@gmail.com}

\end{document}